\documentclass[11pt]{article}

\usepackage[letterpaper,margin=1.1in]{geometry}
\usepackage[T1]{fontenc}
\usepackage{amsmath,amsthm,amssymb}
\usepackage{newpxtext}
\usepackage{newpxmath}
\usepackage{mathtools}
\usepackage{microtype}
\usepackage[hidelinks]{hyperref}
\usepackage[nameinlink,noabbrev]{cleveref}

\numberwithin{equation}{section}
\allowdisplaybreaks

\newtheorem{theorem}{Theorem}[section]
\newtheorem{proposition}[theorem]{Proposition}
\newtheorem{lemma}[theorem]{Lemma}
\newtheorem{corollary}[theorem]{Corollary}
\theoremstyle{definition}
\newtheorem{definition}[theorem]{Definition}

\theoremstyle{remark}

\crefname{theorem}{theorem}{theorems}
\crefname{proposition}{proposition}{propositions}
\crefname{lemma}{lemma}{lemmas}
\crefname{corollary}{corollary}{corollaries}
\crefname{definition}{definition}{definitions}
\crefname{example}{example}{examples}
\crefname{remark}{remark}{remarks}

\newcommand{\SP}{\mathrm{SP}}
\newcommand{\OP}{\mathrm{OP}}

\newcommand{\cQ}{\mathcal{Q}}

\newcommand{\Pf}{\operatorname{Pf}}

\newcommand{\Gam}{\Gamma}
\newcommand{\rhoop}[1]{\rho_{#1}}
\newcommand{\Trel}[2]{T_{#1,#2}}
\newcommand{\diag}{\operatorname{diag}}
\newcommand{\wt}{\operatorname{wt}}

\title{Transition Matrices between Shifted
\texorpdfstring{$t$}{t}-Schur Bases and Cyclotomic Schur
\texorpdfstring{$Q$}{Q}-Positivity}
\author{S.-J. Lee}
\date{}

\begin{document}

\maketitle

\begin{abstract}
For a strict partition $\lambda$, let
$\cQ_\lambda(X;t)=Q_\lambda[X-tX]$ be the shifted
$t$-Schur function arising from the modified Greaves--Jing--Zhu
operator on the odd power-sum ring. We study transition matrices
between the shifted bases with parameters $t$ and $s$. The relative
scaling operator is diagonal in the odd power-sum basis, leading to
explicit spectral data, determinant and trace formulas, weighted
symmetry, a spin-character formula, and a transition Cauchy identity.

For the cyclotomic specialization
$C_{\lambda\mu}^{[M]}(t)=C_{\lambda\mu}(t^M,t)$, the relative
operator becomes plethystic substitution by
$1+t+\cdots+t^{M-1}$. We prove Schur $Q$-positivity and reciprocity,
derive factorization and root-of-unity rank formulas, and give an
exact computation method. For $M=2$, all one-row transitions are
computed explicitly, and the nonzero coefficients are unimodal.
\end{abstract}

\section{Introduction}

Greaves, Jing, and Zhu introduced operator constructions for the
$t$-Schur functions and the $t$-Schur measure \cite{GreavesJingZhu}.
Lee subsequently considered the odd-power-sum analogue of this
construction, with symmetric functions indexed by strict partitions
\cite{Lee}. The main functions used in the present paper are the
shifted $t$-Schur functions
\begin{equation}\label{eq:intro2-pleth}
 \cQ_\lambda(X;t)=Q_\lambda[X-tX],
\end{equation}
where $Q_\lambda$ denotes the Schur $Q$-function attached to the
strict partition $\lambda$. For each parameter $t$, these functions
form a basis of the odd power-sum ring after extending scalars to
$\mathbb Q(t)$. This naturally leads to the question of how the bases
for two different parameters are related.

We define the transition coefficients by
\begin{equation}\label{eq:intro2-transition}
 \cQ_\lambda(X;t)
 =\sum_{\mu\in\SP(|\lambda|)}
  C_{\lambda\mu}(t,s)\cQ_\mu(X;s).
\end{equation}
The relative operator behind this change of basis is
\[
 T_{t,s}=\rho_s^{-1}\rho_t,
 \qquad
 \rho_t(p_n)=(1-t^n)p_n.
\]
Thus
\[
 T_{t,s}(p_n)=\frac{1-t^n}{1-s^n}p_n
 \qquad(n\text{ odd}).
\]
This diagonal action gives a direct way to compute the transition
matrices. In homogeneous degree $N$, their eigenvalues are indexed by
odd partitions $\alpha$ of $N$:
\begin{equation}\label{eq:intro2-eigen}
 r_\alpha(t,s)
 =\prod_{a\in\alpha}\frac{1-t^a}{1-s^a}.
\end{equation}
From this description we derive the cocycle identity, simultaneous
diagonalization, characteristic-polynomial, determinant, trace,
weighted-symmetry, and spin-character formulas. We also include a
transition Cauchy identity which records all transition coefficients
at once.

The cyclotomic specialization is especially concrete. Put
\[
 C_{\lambda\mu}^{[M]}(t)
 :=C_{\lambda\mu}(t^M,t).
\]
Then the transition identity becomes
\begin{equation}\label{eq:intro2-cyclo}
 Q_\lambda[X+tX+\cdots+t^{M-1}X]
 =\sum_\mu C_{\lambda\mu}^{[M]}(t)Q_\mu(X).
\end{equation}
Thus the cyclotomic transition coefficients describe the Schur
$Q$-expansion after substituting the finite alphabet
\[
 X\mapsto X+tX+\cdots+t^{M-1}X.
\]
Using the shifted Littlewood--Richardson rule, we prove
\[
 C_{\lambda\mu}^{[M]}(t)\in\mathbb N[t].
\]
The reversal symmetry of the finite alphabet
$1+t+\cdots+t^{M-1}$ gives the reciprocity formula
\[
 C_{\lambda\mu}^{[M]}(t)
 =t^{(M-1)|\lambda|}C_{\lambda\mu}^{[M]}(t^{-1}),
\]
and hence the coefficients are palindromic.

We also record several structural consequences of the same
specialization. The matrices $C_N^{[M]}(t)$ satisfy a multiplicative
factorization in the integer $M$. Their determinants factor into
cyclotomic polynomials, and their ranks at roots of unity can be read
off from odd partitions. These results give an explicit description
of the transition matrices in each homogeneous degree.

The transition problem considered here is related to other
multiparameter constructions for Schur $Q$-functions, such as
Ivanov's interpolation Schur $Q$-functions \cite{Ivanov}. It is also
compatible with the general viewpoint that changes of variables can be
studied through vertex-operator presentations; see, for example,
\cite{Rozhkovskaya}. The present paper focuses on the diagonal
plethystic scaling in \eqref{eq:intro2-pleth} and on the positivity
properties coming from the finite cyclotomic alphabet.

For $M=2$, we obtain a complete one-row formula. If
$0\leq j<n/2$, then
\begin{equation}\label{eq:intro2-one-row}
 C_{(n),(n-j,j)}^{[2]}(t)
 =t^j(1+t)[n-2j]_t,
 \qquad [r]_t=1+t+\cdots+t^{r-1},
\end{equation}
and transitions to strict partitions of length greater than two
vanish. This gives nonnegative, palindromic, and unimodal
coefficients.

The paper is organized as follows. In
\cref{sec:transition-preliminaries} we review Schur $Q$-functions and
define the shifted bases by diagonal plethystic scaling. General
transition matrices are studied in \cref{sec:transition}. The
cyclotomic specialization is developed in \cref{sec:cyclotomic}.
The one-row formula for $M=2$ appears in \cref{sec:one-row}, and
\cref{sec:computation} gives an exact computational procedure and
low-degree examples.

\section{Schur \texorpdfstring{$Q$}{Q}-functions and shifted bases}
\label{sec:transition-preliminaries}

\subsection{The odd power-sum ring}

A \emph{strict partition} is a finite sequence
$\lambda=(\lambda_1>\cdots>\lambda_\ell>0)$. We write
$|\lambda|=\lambda_1+\cdots+\lambda_\ell$ and
$\ell(\lambda)=\ell$. Let $\SP(N)$ denote the set of strict
partitions of $N$. Let $\OP(N)$ denote the set of partitions of $N$
all of whose parts are odd. Euler's identity gives
$|\SP(N)|=|\OP(N)|$.

Let
\[
 \Gam=\mathbb{Q}[p_1,p_3,p_5,\ldots]
\]
be the odd power-sum ring, graded by $\deg p_n=n$. For a partition
$\alpha=(1^{m_1}3^{m_3}\cdots)\in\OP(N)$, set
\[
 p_\alpha=\prod_i p_{\alpha_i},
 \qquad
 z_\alpha=\prod_{r\geq1}r^{m_r}m_r!.
\]
The elements $p_\alpha$, $\alpha\in\OP(N)$, form a basis of the
homogeneous component $\Gam_N$.

We use the standard Schur $Q$ scalar product
\begin{equation}\label{eq:Q-inner-product}
 \langle p_\alpha,p_\beta\rangle
 =2^{-\ell(\alpha)}z_\alpha\,\delta_{\alpha\beta}.
\end{equation}
The Schur $Q$-functions $Q_\lambda$, $\lambda\in\SP(N)$, form another
basis of $\Gam_N$. We set
\[
 P_\lambda=2^{-\ell(\lambda)}Q_\lambda.
\]
Then
\begin{equation}\label{eq:PQ-duality}
 \langle Q_\lambda,P_\mu\rangle=\delta_{\lambda\mu},
 \qquad
 \langle Q_\lambda,Q_\mu\rangle
 =2^{\ell(\lambda)}\delta_{\lambda\mu}.
\end{equation}
We follow the standard conventions in \cite[Chapter III, Section
8]{Macdonald}; see also \cite{JingSpin,Stembridge}.

The one-row functions $q_r=Q_{(r)}$ are determined by
\begin{equation}\label{eq:one-row-generating}
 Q_X(z)
 :=\sum_{r\geq0}q_r(X)z^r
 =\exp\!\left(
    2\sum_{\substack{n\geq1\\n\text{ odd}}}
      \frac{p_n(X)}{n}z^n
   \right)
 =\prod_i\frac{1+x_i z}{1-x_i z}.
\end{equation}
Here $q_0=1$ and $q_r=0$ for $r<0$. For $r\geq s\geq0$, define
\begin{equation}\label{eq:two-row-classical}
 Q_{(r,s)}
 =q_rq_s+2\sum_{k=1}^{s}(-1)^kq_{r+k}q_{s-k}.
\end{equation}
In particular, $Q_{(r,0)}=q_r$ and $Q_{(r,r)}=0$. If $\lambda$ has
odd length, append a zero part. The Pfaffian Giambelli identity is
\begin{equation}\label{eq:Q-Pfaffian}
 Q_\lambda
 =\Pf\bigl(Q_{(\lambda_i,\lambda_j)}\bigr).
\end{equation}

For later use, define coefficients $\chi_\alpha^\lambda$ by the
power-sum expansion
\begin{equation}\label{eq:spin-expansion}
 Q_\lambda
 =\sum_{\alpha\in\OP(|\lambda|)}
   \frac{2^{\ell(\alpha)}}{z_\alpha}
   \chi_\alpha^\lambda p_\alpha.
\end{equation}
With the usual normalization, these are the spin-character
coefficients. Their representation-theoretic interpretation will not
be needed; only \eqref{eq:spin-expansion} and the orthogonality implied
by \eqref{eq:Q-inner-product} will be used.

\subsection{Diagonal plethystic scaling}

Let $t$ be an indeterminate and extend scalars to $\mathbb Q(t)$.
Define the graded algebra automorphism
\begin{equation}\label{eq:rho-definition}
 \rho_t:\Gam_{\mathbb Q(t)}\longrightarrow\Gam_{\mathbb Q(t)},
 \qquad
 \rho_t(p_n)=(1-t^n)p_n
 \quad(n\text{ odd}).
\end{equation}
Its inverse is given by
\[
 \rho_t^{-1}(p_n)=\frac{1}{1-t^n}p_n.
\]
The shifted $t$-Schur functions associated with the modified odd
Greaves--Jing--Zhu field may be written as
\begin{equation}\label{eq:shifted-basis-definition}
 \cQ_\lambda(X;t)
 :=\rho_tQ_\lambda(X)
 =Q_\lambda[X-tX].
\end{equation}
Indeed, in plethystic notation,
\[
 p_n[X-tX]=(1-t^n)p_n[X].
\]
This is equivalent to the neutral-vertex-operator definition in
\cite{Lee}.

Applying $\rho_t$ to \eqref{eq:spin-expansion} gives
\begin{equation}\label{eq:power-sum-shifted}
 \cQ_\lambda(X;t)
 =\sum_{\alpha\in\OP(|\lambda|)}
   \frac{2^{\ell(\alpha)}}{z_\alpha}
   \chi_\alpha^\lambda
   \prod_{a\in\alpha}(1-t^a)p_\alpha(X).
\end{equation}
Hence $\cQ_\lambda(X;t)$ is polynomial in $t$. Over
$\mathbb Q(t)$, the map $\rho_t$ is invertible, so
\[
 \{\cQ_\lambda(X;t):\lambda\in\SP(N)\}
\]
is a basis of $\Gam_N\otimes_{\mathbb Q}\mathbb Q(t)$.

Every odd power sum is primitive for the standard Hopf structure on
$\Gam$. Therefore $\rho_t$ is a graded Hopf algebra automorphism over
$\mathbb Q(t)$. The multiplication and coproduct structure constants
within a fixed shifted basis are consequently transported from the
classical Schur $Q$ theory. This leads naturally to transition
matrices between two shifted bases with different parameters.

\section{Transition matrices}
\label{sec:transition}

\subsection{Relative diagonal scaling}

Let $t$ and $s$ be independent indeterminates. Define
\begin{equation}\label{eq:T-definition}
 \Trel{t}{s}=\rhoop{s}^{-1}\rhoop{t}.
\end{equation}
Then
\begin{equation}\label{eq:T-on-p}
 \Trel{t}{s}(p_n)
 =r_n(t,s)p_n,
 \qquad
 r_n(t,s)=\frac{1-t^n}{1-s^n}
 \quad(n\text{ odd}).
\end{equation}

\begin{definition}\label{def:transition-coefficients}
For $\lambda\in\SP(N)$, define $C_{\lambda\mu}(t,s)$ by
\begin{equation}\label{eq:C-definition}
 \Trel{t}{s}Q_\lambda
 =\sum_{\mu\in\SP(N)}
  C_{\lambda\mu}(t,s)Q_\mu.
\end{equation}
Let
\[
 C_N(t,s)
 =\bigl(C_{\lambda\mu}(t,s)\bigr)_{
      \lambda,\mu\in\SP(N)},
\]
with both rows and columns ordered in the same fixed order.
\end{definition}

Applying $\rhoop{s}$ to \eqref{eq:C-definition} gives the equivalent
basis-transition identity
\begin{equation}\label{eq:shifted-basis-transition}
 \cQ_\lambda(X;t)
 =\sum_{\mu\in\SP(N)}
  C_{\lambda\mu}(t,s)\cQ_\mu(X;s).
\end{equation}

\begin{proposition}[Cocycle identity]\label{prop:cocycle}
For parameters $t,s,r$,
\begin{equation}\label{eq:cocycle}
 C_N(t,r)=C_N(t,s)C_N(s,r).
\end{equation}
In particular,
\begin{equation}\label{eq:C-inverse}
 C_N(t,t)=I,
 \qquad
 C_N(t,s)^{-1}=C_N(s,t).
\end{equation}
\end{proposition}

\begin{proof}
The operators satisfy
\[
 \Trel{s}{r}\Trel{t}{s}
 =\rhoop{r}^{-1}\rhoop{s}
  \rhoop{s}^{-1}\rhoop{t}
 =\Trel{t}{r}.
\]
Apply both sides to $Q_\lambda$. By
\eqref{eq:C-definition}, the coefficient of $Q_\nu$ on the left is
\[
 \sum_\mu C_{\lambda\mu}(t,s)C_{\mu\nu}(s,r),
\]
which proves \eqref{eq:cocycle}. The statements in
\eqref{eq:C-inverse} follow by setting $r=t$ or $s=t$.
\end{proof}

\subsection{Simultaneous diagonalization}

For $\alpha=(\alpha_1,\ldots,\alpha_k)\in\OP(N)$, set
\begin{equation}\label{eq:r-alpha}
 r_\alpha(t,s)
 =\prod_{i=1}^{k}r_{\alpha_i}(t,s)
 =\prod_{a\in\alpha}\frac{1-t^a}{1-s^a}.
\end{equation}
By \eqref{eq:T-on-p},
\begin{equation}\label{eq:T-palpha}
 \Trel{t}{s}(p_\alpha)=r_\alpha(t,s)p_\alpha.
\end{equation}

\begin{theorem}[Spectral data]\label{thm:spectral-data}
For every $N\geq0$,
\begin{equation}\label{eq:characteristic-polynomial}
 \det\bigl(uI-C_N(t,s)\bigr)
 =\prod_{\alpha\in\OP(N)}
   \bigl(u-r_\alpha(t,s)\bigr).
\end{equation}
Consequently,
\begin{align}
 \det C_N(t,s)
 &=\prod_{\alpha\in\OP(N)}r_\alpha(t,s),
 \label{eq:det-general}\\
 \operatorname{tr}C_N(t,s)
 &=\sum_{\alpha\in\OP(N)}r_\alpha(t,s).
 \label{eq:trace-general}
\end{align}
All matrices $C_N(t,s)$ commute with one another.
\end{theorem}

\begin{proof}
Let $B_N$ be the change-of-basis matrix defined by
\begin{equation}\label{eq:B-definition}
 Q_\lambda
 =\sum_{\alpha\in\OP(N)}
  (B_N)_{\lambda\alpha}p_\alpha.
\end{equation}
Both the $Q_\lambda$ and the $p_\alpha$ are bases of $\Gam_N$, so
$B_N$ is invertible. Let
\[
 D_N(t,s)=\diag\bigl(r_\alpha(t,s):\alpha\in\OP(N)\bigr).
\]
Equations \eqref{eq:C-definition}, \eqref{eq:T-palpha}, and
\eqref{eq:B-definition} imply
\[
 C_N(t,s)B_N=B_ND_N(t,s),
\]
and hence
\begin{equation}\label{eq:C-similarity}
 C_N(t,s)=B_ND_N(t,s)B_N^{-1}.
\end{equation}
The characteristic polynomial, determinant, and trace follow. Since
all diagonal matrices $D_N(t,s)$ commute, their conjugates by the same
matrix $B_N$ also commute.
\end{proof}

\begin{corollary}[Trace generating function]\label{cor:trace-generating}
As a formal power series in $z$,
\begin{equation}\label{eq:trace-generating}
 \sum_{N\geq0}\operatorname{tr}C_N(t,s)z^N
 =\prod_{\substack{n\geq1\\n\text{ odd}}}
  \frac{1}{1-r_n(t,s)z^n}.
\end{equation}
\end{corollary}

\begin{proof}
By \eqref{eq:trace-general}, the coefficient of $z^N$ on the left is
the sum of $\prod_{a\in\alpha}r_a(t,s)$ over odd partitions of $N$.
The product on the right independently records the multiplicity of
each odd part.
\end{proof}

\subsection{Weighted symmetry and spin-character formula}

\begin{theorem}[Weighted symmetry]\label{thm:weighted-symmetry}
For strict partitions $\lambda,\mu$ of the same size,
\begin{equation}\label{eq:weighted-symmetry}
 2^{\ell(\mu)}C_{\lambda\mu}(t,s)
 =2^{\ell(\lambda)}C_{\mu\lambda}(t,s).
\end{equation}
Equivalently, if
\[
 H_N=\diag\bigl(2^{\ell(\lambda)}:\lambda\in\SP(N)\bigr),
\]
then $C_N(t,s)H_N$ is symmetric. Equivalently, after adjoining
$\sqrt{2}$ if desired, $H_N^{-1/2}C_N(t,s)H_N^{1/2}$ is symmetric.
\end{theorem}

\begin{proof}
The operator $\Trel{t}{s}$ is diagonal in the orthogonal basis
$\{p_\alpha\}$ for the scalar product
\eqref{eq:Q-inner-product}. Hence it is self-adjoint:
\[
 \langle \Trel{t}{s}f,g\rangle
 =\langle f,\Trel{t}{s}g\rangle.
\]
Using \eqref{eq:C-definition} and \eqref{eq:PQ-duality},
\[
 \langle \Trel{t}{s}Q_\lambda,Q_\mu\rangle
 =2^{\ell(\mu)}C_{\lambda\mu}(t,s),
\]
whereas
\[
 \langle Q_\lambda,\Trel{t}{s}Q_\mu\rangle
 =2^{\ell(\lambda)}C_{\mu\lambda}(t,s).
\]
This proves \eqref{eq:weighted-symmetry}; the matrix statements are
immediate reformulations.
\end{proof}

\begin{theorem}[Spin-character formula]\label{thm:spin-character-formula}
With the coefficients from \eqref{eq:spin-expansion},
\begin{equation}\label{eq:spin-character-transition}
 C_{\lambda\mu}(t,s)
 =2^{-\ell(\mu)}
  \sum_{\alpha\in\OP(N)}
   \frac{2^{\ell(\alpha)}}{z_\alpha}
   \chi_\alpha^\lambda\chi_\alpha^\mu
   r_\alpha(t,s).
\end{equation}
\end{theorem}

\begin{proof}
Since $P_\mu$ is dual to $Q_\mu$,
\[
 C_{\lambda\mu}(t,s)
 =\langle \Trel{t}{s}Q_\lambda,P_\mu\rangle
 =2^{-\ell(\mu)}
  \langle \Trel{t}{s}Q_\lambda,Q_\mu\rangle.
\]
Insert the expansions \eqref{eq:spin-expansion}, use
\eqref{eq:T-palpha}, and then apply
\eqref{eq:Q-inner-product}. Only equal odd partitions survive, and
the scalar factor is
\[
 \frac{2^{\ell(\alpha)}}{z_\alpha}
 \frac{2^{\ell(\alpha)}}{z_\alpha}
 2^{-\ell(\alpha)}z_\alpha
 =\frac{2^{\ell(\alpha)}}{z_\alpha}.
\]
This gives \eqref{eq:spin-character-transition}.
\end{proof}

\subsection{A transition Cauchy identity}

\begin{theorem}\label{thm:transition-Cauchy}
For two alphabets $X$ and $Y$,
\begin{equation}\label{eq:transition-Cauchy}
 \sum_{\lambda\in\SP}
 \sum_{\mu\in\SP(|\lambda|)}
 C_{\lambda\mu}(t,s)Q_\mu(X)P_\lambda(Y)
 =\exp\!\left(
   2\sum_{\substack{n\geq1\\n\text{ odd}}}
    \frac{1-t^n}{1-s^n}
    \frac{p_n(X)p_n(Y)}{n}
  \right).
\end{equation}
\end{theorem}

\begin{proof}
The classical Schur $Q$ Cauchy identity is
\begin{equation}\label{eq:classical-Q-Cauchy}
 \sum_{\lambda\in\SP}Q_\lambda(X)P_\lambda(Y)
 =\exp\!\left(
   2\sum_{\substack{n\geq1\\n\text{ odd}}}
    \frac{p_n(X)p_n(Y)}{n}
  \right).
\end{equation}
Apply $\Trel{t}{s}$ to the $X$ variables. The left side becomes the
left side of \eqref{eq:transition-Cauchy} by
\eqref{eq:C-definition}; the right side follows from
\eqref{eq:T-on-p}.
\end{proof}

\section{Cyclotomic transition coefficients}
\label{sec:cyclotomic}

Fix a positive integer $M$ and define the finite plethystic alphabet
\begin{equation}\label{eq:AM-definition}
 A_M(t)=1+t+\cdots+t^{M-1}.
\end{equation}
We call a polynomial $f(t)$ \emph{palindromic of darga $d$} if
$f(t)=t^df(t^{-1})$; this convention allows zero coefficients at both
ends of the ambient degree interval $[0,d]$. For every $n\geq1$,
\begin{equation}\label{eq:power-AM}
 p_n[A_M(t)]
 =1+t^n+\cdots+t^{(M-1)n}
 =\frac{1-t^{Mn}}{1-t^n}.
\end{equation}

\subsection{Cyclotomic transition coefficients}

\begin{definition}\label{def:cyclotomic-transition}
Set
\begin{equation}\label{eq:CM-definition}
 C_{\lambda\mu}^{[M]}(t)
 :=C_{\lambda\mu}(t^M,t),
 \qquad
 C_N^{[M]}(t):=C_N(t^M,t).
\end{equation}
\end{definition}

By \eqref{eq:T-on-p} and \eqref{eq:power-AM},
\begin{equation}\label{eq:TM-on-p}
 \Trel{t^M}{t}(p_n)
 =\bigl(1+t^n+\cdots+t^{(M-1)n}\bigr)p_n.
\end{equation}
Therefore, for every $f\in\Gam$,
\begin{equation}\label{eq:TM-plethystic}
 \Trel{t^M}{t}f(X)=f[A_M(t)X].
\end{equation}
The transition identity becomes
\begin{equation}\label{eq:CM-transition}
 Q_\lambda[X+tX+\cdots+t^{M-1}X]
 =\sum_{\mu\in\SP(|\lambda|)}
  C_{\lambda\mu}^{[M]}(t)Q_\mu(X).
\end{equation}

\begin{theorem}[Positivity]\label{thm:CM-positivity}
For all strict partitions $\lambda$ and $\mu$ of the same size,
\begin{equation}\label{eq:CM-positive}
 C_{\lambda\mu}^{[M]}(t)\in\mathbb{N}[t].
\end{equation}
\end{theorem}

\begin{proof}
For alphabets $U$ and $V$, the Schur $Q$ coproduct identity is
\begin{equation}\label{eq:skew-coproduct}
 Q_\lambda[U+V]
 =\sum_{\nu\subseteq\lambda}
   Q_\nu[U]Q_{\lambda/\nu}[V].
\end{equation}
Iterating \eqref{eq:skew-coproduct} with
$U_j=t^jX$ gives
\begin{equation}\label{eq:chain-expansion}
 \begin{aligned}
 &Q_\lambda[X+tX+\cdots+t^{M-1}X]\\
 &\quad=
 \sum_{\varnothing=\lambda^{(0)}\subseteq
       \lambda^{(1)}\subseteq\cdots\subseteq
       \lambda^{(M)}=\lambda}
 t^{\wt(\lambda^\bullet)}
 \prod_{j=0}^{M-1}
 Q_{\lambda^{(j+1)}/\lambda^{(j)}}(X),
 \end{aligned}
\end{equation}
where
\begin{equation}\label{eq:chain-weight}
 \wt(\lambda^\bullet)
 =\sum_{j=0}^{M-1}
  j\bigl(|\lambda^{(j+1)}|-|\lambda^{(j)}|\bigr).
\end{equation}
The shifted Littlewood--Richardson rule implies that every skew Schur
$Q$-function is Schur $Q$-positive and that products of Schur
$Q$-functions are Schur $Q$-positive; see
\cite{Stembridge}. Hence each product in
\eqref{eq:chain-expansion} is a nonnegative integral linear
combination of $Q_\mu(X)$. The weight
$t^{\wt(\lambda^\bullet)}$ is a monomial with nonnegative exponent.
Comparing with \eqref{eq:CM-transition} proves
\eqref{eq:CM-positive}.
\end{proof}

\begin{theorem}[Reciprocity]\label{thm:CM-reciprocity}
If $\lambda,\mu\in\SP(N)$, then
\begin{equation}\label{eq:CM-reciprocity}
 C_{\lambda\mu}^{[M]}(t)
 =t^{(M-1)N}C_{\lambda\mu}^{[M]}(t^{-1}).
\end{equation}
Consequently, $C_{\lambda\mu}^{[M]}(t)$ is palindromic with darga
$(M-1)N$. Moreover,
\begin{equation}\label{eq:CM-end-coefficients}
 C_{\lambda\mu}^{[M]}(0)=\delta_{\lambda\mu},
 \qquad
 [t^{(M-1)N}]C_{\lambda\mu}^{[M]}(t)
 =\delta_{\lambda\mu}.
\end{equation}
\end{theorem}

\begin{proof}
The alphabet identity
\[
 A_M(t)=t^{M-1}A_M(t^{-1})
\]
and homogeneity of $Q_\lambda$ give
\[
 Q_\lambda[A_M(t)X]
 =t^{(M-1)N}Q_\lambda[A_M(t^{-1})X].
\]
Expand both sides in the $Q$-basis using
\eqref{eq:CM-transition} and compare coefficients. This proves
\eqref{eq:CM-reciprocity}.

At $t=0$, the alphabet $A_M(0)$ consists of the single letter $1$, so
$Q_\lambda[A_M(0)X]=Q_\lambda[X]$. Thus the constant term is
$\delta_{\lambda\mu}$. The highest coefficient follows from
reciprocity.
\end{proof}

\begin{corollary}[Weighted symmetry]\label{cor:CM-weighted}
For $\lambda,\mu\in\SP(N)$,
\begin{equation}\label{eq:CM-weighted}
 2^{\ell(\mu)}C_{\lambda\mu}^{[M]}(t)
 =2^{\ell(\lambda)}C_{\mu\lambda}^{[M]}(t).
\end{equation}
\end{corollary}

\begin{proof}
Specialize \cref{thm:weighted-symmetry} at $(t,s)=(t^M,t)$.
\end{proof}

\subsection{Factorization, determinant, and roots of unity}

Let $\mathscr{T}_M(t)$ denote the plethystic operator
\[
 \mathscr{T}_M(t)f(X)=f[A_M(t)X].
\]

\begin{proposition}[Multiplicative factorization]\label{prop:CM-factorization}
For positive integers $M$ and $L$,
\begin{equation}\label{eq:alphabet-factorization}
 A_{ML}(t)=A_M(t)A_L(t^M)
\end{equation}
plethystically, and
\begin{equation}\label{eq:operator-factorization}
 \mathscr{T}_{ML}(t)
 =\mathscr{T}_L(t^M)\mathscr{T}_M(t).
\end{equation}
Consequently,
\begin{equation}\label{eq:matrix-factorization}
 C_N^{[ML]}(t)
 =C_N^{[M]}(t)C_N^{[L]}(t^M)
 =C_N^{[L]}(t^M)C_N^{[M]}(t).
\end{equation}
\end{proposition}

\begin{proof}
The ordinary polynomial identity
\[
 (1+t+\cdots+t^{M-1})
 (1+t^M+\cdots+t^{(L-1)M})
 =1+t+\cdots+t^{ML-1}
\]
is precisely \eqref{eq:alphabet-factorization}. On power sums,
\[
 [M]_{t^n}[L]_{t^{Mn}}=[ML]_{t^n},
\]
so \eqref{eq:operator-factorization} follows. Applying both sides to
$Q_\lambda$ and using the row-by-row convention in
\eqref{eq:C-definition} gives the first matrix product in
\eqref{eq:matrix-factorization}. The second equality follows because
all transition matrices are simultaneously diagonalizable by
\cref{thm:spectral-data}.
\end{proof}

\begin{theorem}[Cyclotomic determinant]\label{thm:CM-determinant}
For $N\geq0$,
\begin{equation}\label{eq:CM-determinant}
 \det C_N^{[M]}(t)
 =\prod_{\alpha\in\OP(N)}
   \prod_{a\in\alpha}[M]_{t^a},
\end{equation}
where
\[
 [M]_x=1+x+\cdots+x^{M-1}.
\]
Equivalently,
\begin{equation}\label{eq:CM-cyclotomic-factorization}
 \det C_N^{[M]}(t)
 =\prod_{\alpha\in\OP(N)}
   \prod_{a\in\alpha}
   \prod_{\substack{d\mid M\\d>1}}
    \Phi_d(t^a),
\end{equation}
where $\Phi_d$ is the $d$th cyclotomic polynomial.
\end{theorem}

\begin{proof}
Specialize \eqref{eq:det-general} at $(t,s)=(t^M,t)$ and use
\[
 \frac{1-t^{Ma}}{1-t^a}=[M]_{t^a}.
\]
The factorization in \eqref{eq:CM-cyclotomic-factorization} follows
from
\[
 [M]_x=\frac{1-x^M}{1-x}
 =\prod_{\substack{d\mid M\\d>1}}\Phi_d(x).
\]
\end{proof}

\begin{theorem}[Root-of-unity rank]\label{thm:root-rank}
Let $\zeta\in\mathbb{C}$ be a root of unity. Then
\begin{equation}\label{eq:root-rank}
 \operatorname{rank}C_N^{[M]}(\zeta)
 =\#\left\{
  \alpha\in\OP(N):
  [M]_{\zeta^a}\neq0
  \text{ for every part }a\text{ of }\alpha
 \right\}.
\end{equation}
In particular, for $N>0$,
\begin{equation}\label{eq:t-minus-one}
 C_N^{[M]}(-1)
 =\begin{cases}
   I,&M\text{ odd},\\
   0,&M\text{ even}.
  \end{cases}
\end{equation}
\end{theorem}

\begin{proof}
The change-of-basis matrix $B_N$ in
\eqref{eq:C-similarity} is independent of $t$ and invertible over
$\mathbb{C}$. After the cyclotomic specialization, the diagonal entry
indexed by $\alpha$ is
\[
 \prod_{a\in\alpha}[M]_{\zeta^a}.
\]
Thus the rank equals the number of nonzero diagonal entries, proving
\eqref{eq:root-rank}.

At $\zeta=-1$, every part $a$ of an odd partition is odd, so
$\zeta^a=-1$. Now $[M]_{-1}=1$ if $M$ is odd and $0$ if $M$ is
even. Hence the diagonal operator is the identity in the first case
and zero in every positive degree in the second case. Conjugating by
$B_N$ proves \eqref{eq:t-minus-one}.
\end{proof}

\section{One-row transitions for \texorpdfstring{$M=2$}{M=2}}
\label{sec:one-row}

We now set $M=2$, so
\[
 A_2(t)=1+t.
\]
The transition identity for a one-row Schur $Q$-function is
\begin{equation}\label{eq:q-n-double-alphabet}
 q_n[X+tX]
 =\sum_{b=0}^{n}t^bq_{n-b}(X)q_b(X),
\end{equation}
which follows immediately by multiplying the two generating series
$Q_X(z)$ and $Q_{tX}(z)$.

\subsection{An inverse two-row identity}

\begin{lemma}\label{lem:inverse-two-row}
For integers $r\geq s\geq0$,
\begin{equation}\label{eq:inverse-two-row}
 q_rq_s
 =Q_{(r,s)}+2\sum_{k=1}^{s}Q_{(r+k,s-k)},
\end{equation}
with the conventions $Q_{(r,0)}=q_r$ and $Q_{(r,r)}=0$.
\end{lemma}

\begin{proof}
Insert \eqref{eq:two-row-classical} into the right side of
\eqref{eq:inverse-two-row}. The term $q_rq_s$ occurs once, from
$Q_{(r,s)}$. Fix $j\geq1$. The coefficient of
$q_{r+j}q_{s-j}$ is
\[
 2+2(-1)^j+4\sum_{h=1}^{j-1}(-1)^h.
\]
If $j$ is odd, the finite sum is $0$, and the displayed coefficient is
$2-2=0$. If $j$ is even, the finite sum is $-1$, and the coefficient
is $2+2-4=0$. Hence every term other than $q_rq_s$ cancels.
\end{proof}

\subsection{Closed one-row formula}

\begin{theorem}\label{thm:one-row-M2}
Let $n\geq1$. For every integer $j$ with $0\leq j<n/2$,
\begin{equation}\label{eq:one-row-M2}
 C_{(n),(n-j,j)}^{[2]}(t)
 =t^j(1+t)[n-2j]_t,
\end{equation}
where $[r]_t=1+t+\cdots+t^{r-1}$. If
$\mu\in\SP(n)$ has length greater than two, then
\begin{equation}\label{eq:one-row-length-vanishing}
 C_{(n),\mu}^{[2]}(t)=0.
\end{equation}
Thus
\begin{equation}\label{eq:one-row-full-expansion}
 Q_{(n)}[X+tX]
 =\sum_{0\leq j<n/2}
  t^j(1+t)[n-2j]_t
  Q_{(n-j,j)}(X),
\end{equation}
where $Q_{(n,0)}=Q_{(n)}$.
\end{theorem}

\begin{proof}
Pair the terms indexed by $b$ and $n-b$ in
\eqref{eq:q-n-double-alphabet}. Every product can be written with the
larger index first. By \cref{lem:inverse-two-row}, each product
$q_{n-b}q_b$ expands only into Schur $Q$-functions of length at most
two. This proves \eqref{eq:one-row-length-vanishing}.

Fix $j<n/2$. In the expansion of $q_{n-b}q_b$ with
$0\leq b\leq\lfloor n/2\rfloor$, the function
$Q_{(n-j,j)}$ occurs with coefficient $1$ when $b=j$, with coefficient
$2$ when $b>j$, and not at all when $b<j$. Restoring the paired
weights $t^b+t^{n-b}$, and treating the middle term in the even case in
the same way, gives
\[
 \begin{aligned}
 C_{(n),(n-j,j)}^{[2]}(t)
 &=t^j+2t^{j+1}+2t^{j+2}+\cdots
   +2t^{n-j-1}+t^{n-j}\\
 &=t^j(1+t)(1+t+\cdots+t^{n-2j-1})\\
 &=t^j(1+t)[n-2j]_t.
 \end{aligned}
\]
This proves \eqref{eq:one-row-M2} and
\eqref{eq:one-row-full-expansion}.
\end{proof}

\begin{corollary}\label{cor:one-row-unimodal}
Every nonzero transition coefficient $C_{(n),\mu}^{[2]}(t)$ is
nonnegative, palindromic, and unimodal.
\end{corollary}

\begin{proof}
The nonzero coefficients have the form
\[
 t^j(1+t)[n-2j]_t.
\]
If $n-2j=1$, this is $t^j+t^{j+1}$. If $n-2j\geq2$, its coefficient
sequence on the interval from degree $j$ to degree $n-j$ is
\[
 1,2,2,\ldots,2,1.
\]
In both cases the sequence is unimodal and symmetric about degree
$n/2$.
\end{proof}

\section{Exact computation and low-degree examples}
\label{sec:computation}

This section records a direct algorithm for computing the transition
matrices. It uses only recurrences, Pfaffians, and rational linear
algebra.

\subsection{Computing the Schur \texorpdfstring{$Q$}{Q}-to-power-sum matrix}

Differentiate \eqref{eq:one-row-generating} logarithmically:
\[
 zQ_X'(z)
 =2\left(
   \sum_{\substack{j\geq1\\j\text{ odd}}}p_jz^j
  \right)Q_X(z).
\]
Comparing coefficients gives the recurrence
\begin{equation}\label{eq:q-recurrence-computation}
 nq_n
 =2\sum_{\substack{1\leq j\leq n\\j\text{ odd}}}
   p_jq_{n-j},
 \qquad q_0=1.
\end{equation}
Starting from $q_0$, this computes every one-row function in the odd
power-sum basis. Next compute $Q_{(r,s)}$ from
\eqref{eq:two-row-classical}, and compute general $Q_\lambda$ using the
Pfaffian identity \eqref{eq:Q-Pfaffian}.

Fix $N$. Order $\SP(N)$ and $\OP(N)$, and form the square matrix
$B_N$ from
\[
 Q_\lambda
 =\sum_{\alpha\in\OP(N)}
  (B_N)_{\lambda\alpha}p_\alpha.
\]
For general parameters, form the diagonal matrix
\begin{equation}\label{eq:D-computation}
 D_N(t,s)
 =\diag\left(
   \prod_{a\in\alpha}\frac{1-t^a}{1-s^a}
   :\alpha\in\OP(N)
  \right).
\end{equation}
Then
\begin{equation}\label{eq:C-computation}
 C_N(t,s)=B_ND_N(t,s)B_N^{-1}.
\end{equation}
For the cyclotomic matrix, replace the diagonal entry in
\eqref{eq:D-computation} by
\begin{equation}\label{eq:DM-computation}
 \prod_{a\in\alpha}[M]_{t^a}.
\end{equation}
Equations \eqref{eq:C-computation} and
\eqref{eq:DM-computation} provide an exact symbolic computation over
$\mathbb{Q}(t,s)$ or $\mathbb{Q}[t]$.

The following identities are useful consistency checks:
\begin{align*}
 C_N(t,s)C_N(s,r)&=C_N(t,r),\\
 C_N(t,s)^{-1}&=C_N(s,t),\\
 2^{\ell(\mu)}C_{\lambda\mu}(t,s)
 &=2^{\ell(\lambda)}C_{\mu\lambda}(t,s),\\
 \det C_N(t,s)&=\prod_{\alpha\in\OP(N)}r_\alpha(t,s),\\
 C_N^{[ML]}(t)&=C_N^{[M]}(t)C_N^{[L]}(t^M).
\end{align*}

\subsection{Degree three}

Order the strict partitions of $3$ as
\[
 (3),(2,1).
\]
Set
\[
 r_1=\frac{1-t}{1-s},
 \qquad
 r_3=\frac{1-t^3}{1-s^3}.
\]
Using
\[
 Q_{(3)}=\frac{4}{3}p_1^3+\frac{2}{3}p_3,
 \qquad
 Q_{(2,1)}=\frac{4}{3}p_1^3-\frac{4}{3}p_3,
\]
one obtains
\begin{equation}\label{eq:C3-general}
 C_3(t,s)
 =\begin{pmatrix}
 \dfrac{2r_1^3+r_3}{3}
 &\dfrac{r_1^3-r_3}{3}\\[3mm]
 \dfrac{2(r_1^3-r_3)}{3}
 &\dfrac{r_1^3+2r_3}{3}
 \end{pmatrix}.
\end{equation}
Its eigenvalues are $r_1^3$ and $r_3$, corresponding to the odd
partitions $(1,1,1)$ and $(3)$. Its determinant is $r_1^3r_3$, and
the off-diagonal entries satisfy the weighted symmetry
\[
 4C_{(3),(2,1)}(t,s)
 =2C_{(2,1),(3)}(t,s).
\]

For $M=2$, substitute $r_1=1+t$ and $r_3=1+t^3$:
\begin{equation}\label{eq:C3-M2}
 C_3^{[2]}(t)
 =\begin{pmatrix}
 1+2t+2t^2+t^3&t+t^2\\
 2t+2t^2&1+t+t^2+t^3
 \end{pmatrix}.
\end{equation}
Every entry is nonnegative and palindromic with darga $3$.

\subsection{Degree four}

Order $\SP(4)$ as $(4),(3,1)$. Direct computation gives
\begin{equation}\label{eq:C4-M2}
 C_4^{[2]}(t)
 =\begin{pmatrix}
 1+2t+2t^2+2t^3+t^4
 &t+2t^2+t^3\\
 2t+4t^2+2t^3
 &1+3t+4t^2+3t^3+t^4
 \end{pmatrix}.
\end{equation}
The first row agrees with \cref{thm:one-row-M2}. The second row is
forced from the first off-diagonal entry by weighted symmetry, while
the remaining diagonal entry may be obtained from the trace or by the
direct matrix calculation \eqref{eq:C-computation}.

\end{document}